\documentclass[10pt,reqno,a4paper]{amsart}

\usepackage[pdftex]{graphicx}
\usepackage{amsmath}
\usepackage{amssymb,amsthm,hyperref,caption}
\usepackage{color}
\definecolor{meinBlau}{rgb}{0,0.2,0.65} 
\definecolor{blau}{rgb}{0,0,0.75} 
\definecolor{rot}{rgb}{0.74,0,0} 
\hypersetup{colorlinks,linkcolor=blau,citecolor=blue,urlcolor=meinBlau}
\usepackage{times}
\usepackage{enumitem}

\allowdisplaybreaks

\newtheorem{theorem}{Theorem}
\newtheorem{lemma}[theorem]{Lemma}

\newtheorem{prop}{Proposition}

\theoremstyle{definition}
\newtheorem{urn}{Urn}
\newtheorem{remark}{Remark}
\newtheorem{example}{Example}
\newtheorem{defi}{Definition}

\def\P{{\mathbb {P}}}
\def\E{{\mathbb {E}}}

\newcommand{\impli}{\ensuremath{\ \Rightarrow\ }}

\newcommand{\mt}[1]{{\textsc{#1}}}

\newcommand{\dit}{\ensuremath{(b,d)}\text{-ary ITs}}
\newcommand{\port}{\ensuremath{(b,\alpha)}\text{-PORTs}}
\newcommand{\Dit}{\ensuremath{(b,d)}\text{-ary} \text{in\-crea\-sing} \text{trees}}
\newcommand{\Port}{\ensuremath{(b,\alpha)}\text{-plane} \text{oriented} \text{recursive} \text{trees}}
\newcommand{\Rec}{\text{bucket recursive trees}}

\newcommand{\fallfak}[2]{\ensuremath{#1^{\underline{#2}}}}

\newcommand{\N}{\ensuremath{\mathbb{N}}}
\newcommand{\Q}{\ensuremath{\mathbb{Q}}}

\DeclareMathOperator{\grad}{deg}

\newcommand{\Cc}{\ensuremath{\frac{c_1}{c_2}}}
\newcommand{\cC}{\ensuremath{\frac{c_2}{c_1}}}

\DeclareMathOperator{\law}{\overset{\mathcal{L}}{=}}
\DeclareMathOperator{\claw}{\overset{\mathcal{L}}{\rightarrow}}

\begin{document}

\author[M.~Kuba]{Markus Kuba}
\address{Markus Kuba\\
Department Applied Mathematics and Physics\\
University of Applied Sciences - Technikum Wien\\
H\"ochst\"adtplatz 5, 1200 Wien} %
\email{kuba@technikum-wien.at}

\author[A.~Panholzer]{Alois Panholzer}
\address{Alois Panholzer\\
Institut f{\"u}r Diskrete Mathematik und Geometrie\\
Technische Universit\"at Wien\\
Wiedner Hauptstr. 8-10/104\\
1040 Wien, Austria} \email{Alois.Panholzer@tuwien.ac.at}

\title[Tree evolution processes for bucket increasing trees]{Tree evolution processes for bucket increasing trees}


\begin{abstract}
We provide a fundamental result for bucket increasing trees, which gives a complete characterization of all families of bucket increasing trees that can be generated by a tree evolution process. We also provide several equivalent properties, complementing and extending earlier results for ordinary increasing trees to bucket trees. Additionally, we state second order results
for the number of descendants of label $j$, again extending earlier results in the literature.
\end{abstract}

\maketitle

\section{Introduction}
Increasing trees of size $n$ are rooted labelled trees with label set $[n]=\{1,2\dots,n\}$, which have the property that the labels are increasing along any path from the root to a leaf. 
An important and fundamental result for increasing trees is the characterization of such tree families, which can be generated by a tree evolution process. Increasing trees and in particular those families generated by tree evolution processes are of great importance in applications, as they are used to describe the spread of epidemics, to model pyramid schemes, as a model of stochastic processes like the Chinese restaurant process or P\'olya-Eggenberger urn models, and as a growth model of the world wide web. It turns out~\cite{PanPro2007} that only three different families, namely recursive trees, $d$-ary increasing trees ($d\in\N\setminus\{1\}$), and generalized plane-oriented recursive trees\footnote{The latter families, as well as closely related trees, also appear in the literature under the names preferential attachment trees~\cite{AddarioA,BA}, nonuniform random recursive trees~\cite{JerzyA}, heap-ordered trees~\cite{MPP,Prod1996}, or scale-free trees~\cite{Bollo2,Bollo1}.} can be constructed by a tree evolution process. These three tree families are sometimes grouped under the umbrella \emph{grown families of increasing trees} or \emph{very simple increasing trees}. They are studied in a myriad of articles including the recent studies~\cite{AddarioB,BertoinA,BodiniA,WagnerB,FuchsA,FuchsB,JansonA,JansonB,HenningA,WagnerA} and they are also intimately connected to urn models~\cite{AddarioA,Janson2005RSA,PanPro2007}. We also refer to the book of Drmota~\cite{Drmota} and the many references therein.

\smallskip

In this work we are concerned with so-called bucket increasing trees, which are multilabelled generalizations of increasing trees. All vertices $v$ of a tree $T$ are considered as buckets having a maximal capacity of $b\in\N$ labels. The integer $c=c(v)$ denotes the current capacity or current load of a node $v\in T$, with $1 \le c \le b$. Additionally, for bucket increasing trees we assume that only fully saturated nodes, i.e., nodes $v$ with $c(v)=b$, may have an out-degree greater than zero, thus all non-leaves are saturated. Leaves might be either saturated or unsaturated. The size $|T|$ of a tree $T$ is always given by the total number of labels in the tree, or equivalently, by the sum of the current capacities of the nodes. Apparently, for $b=1$ we obtain ordinary increasing trees where a single bucket or node may only hold a single label. As our main result we provide a fundamental characterization of bucket increasing trees. We prove, that only three tree (parameterized) families can be constructed by a tree evolution process, thus generalizing a result of~\cite{PanPro2007}. Moreover, we provide five additional equivalent properties of such evolving bucket increasing trees. This generalizes results of~\cite{PK2008,BucketPanKu2009,BucketPanKuAccepted,MahSmy1995,Pan2006,PanPro2007}. For the reader's convenience and the sake of completeness, we also collect and unify arguments of~\cite{BucketPanKu2009,BucketPanKuAccepted}. Furthermore, we generalize several results of~\cite{IncDesc,BucketPanKu2009,BucketPanKuAccepted} concerning the limit laws of the number of descendants of label $j$, providing second order results, in other words, a limit law for the number of descendants centered by its almost-sure beta limit law.

\section{Families of bucket increasing trees}
In the following we present the general combinatorial model of bucket increasing trees. Then, we describe the three 
different tree evolution processes, which generate random bucket increasing trees in a step-by-step fashion.
We also define the corresponding combinatorial models of such evolving bucket increasing trees.

\subsection{Combinatorial description of bucket increasing tree families}
Our presentation follows~\cite{BucketPanKu2009}. Our basic objects are rooted ordered trees, i.e., trees where the order of the subtrees of any node is of relevance. Each node $v$ of such a tree is a bucket with an integer capacity $c = c(v)$, with $1 \le c \le b$, for a given
maximal integer bucket size $b \ge 1$. We assume that all internal nodes (i.e., non-leaves) in the tree must be saturated ($c=b$), while the leaves might be either saturated or unsaturated ($c<b$). A tree $T$ defined in this way is called a bucket ordered tree with maximal bucket size $b$; let us denote by $\mathcal{B} = \mathcal{B}_{b}$ the family of all bucket ordered trees with maximal bucket size $b$.

As already mentioned before, for bucket ordered trees we define the size $|T|$ of a tree $T$ via $|T| = \sum_{v} c(v)$, where $c(v)$ ranges over
all vertices of $T$. An increasing labelling $\ell(T)$ of a bucket ordered tree $T$ is then a labelling of $T$, where the labels $\{1, 2,
\dots, |T|\}$ are distributed amongst the nodes of $T$, such that the following conditions are satisfied: $(i)$ every node $v$
contains exactly $c(v)$ labels, $(ii)$ the labels within a node are arranged in increasing order, $(iii)$ each sequence of labels along
any path starting at the root is increasing. A bucket ordered increasing tree $\tilde{T}$ is then given by a pair $\tilde{T}=(T,\ell(T))$. 

Then a class $\mathcal{T}$ of a family of bucket increasing trees with maximal bucket size $b$ can be defined
in the following way. A sequence of non-negative numbers $(\varphi_{k})_{k \ge 0}$ with $\varphi_{0} > 0$
and a sequence of non-negative numbers $\psi_{1}, \psi_{2}, \dots, \psi_{b-1}$
is used to define the weight $w(T)$ of any bucket ordered tree $T$ by $w(T) := \prod_{v} w(v)$, where $v$ ranges over all vertices of $T$. The weight $w(v)$ of a node $v$ is given as follows, where $\grad(v)$ denotes the out-degree (i.e., the number of children) of node $v$:
\begin{equation*}
   w(v) =
   \begin{cases}
      \varphi_{\grad(v)}, & \quad \text{if} \enspace c(v)=b, \\
      \psi_{c(v)}, & \quad \text{if} \enspace c(v) < b.
   \end{cases}
\end{equation*}
Thus, for saturated nodes the weight depends on the out-degree $\grad(v)$ and is described by the sequence $(\varphi_{k})_{k \ge 0}$, whereas for unsaturated nodes the weight depends on the capacity $c(v)$ and is described by the sequence $(\psi_{k})_{1 \le k \le b-1}$.

Furthermore, $\mathcal{L}(T)$ denotes the set of different increasing labellings $\ell(T)$ of the tree $T$ with distinct integers $\{1, 2, \dots, |T|\}$, where $L(T) := \big|\mathcal{L}(T)\big|$ denotes its cardinality.
The family $\mathcal{T}$ consists of all trees $\tilde{T}=(T,\ell(T))$, with their weights $w(T)$ and the set of increasing labellings $\mathcal{L}(T)$, and we define $w(\tilde{T}):=w(T)$. Concerning bucket ordered increasing trees, note that the left-to-right order of the subtrees of the nodes is relevant. E.g., the trees \raisebox{-0.6em}{\includegraphics[height=2em]{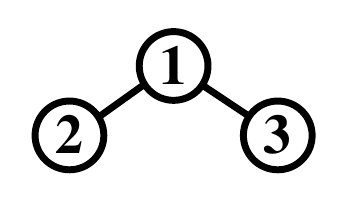}} and
\raisebox{-0.6em}{\includegraphics[height=2em]{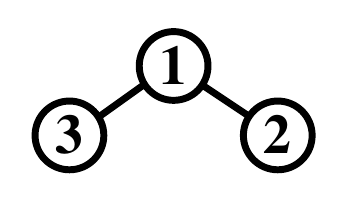}} are forming two different trees.

For a given degree-weight sequence $(\varphi_{k})_{k \ge 0}$ with a degree-weight generating function
$\varphi(t) := \sum_{k \ge 0} \varphi_{k} t^{k}$ and a bucket-weight sequence $\psi_{1}, \dots, \psi_{b-1}$,
we define now the total weights $T_n$ of such size-$n$ bucket ordered increasing trees by 
\[
T_{n} :=  \sum_{T\in\mathcal{B}\colon |T|=n}w(T)\cdot L(T)=\sum_{\tilde{T}=(T,\ell(T))\in\mathcal{T}\colon |T|=n} w(\tilde{T}).
\]

\medskip

It is advantageous for such enumeration problems to describe a family of increasing trees $\mathcal{T}$ by the following
formal recursive equation:
\begin{align}
   \mathcal{T} & = \psi_{1} \cdot \raisebox{-0.5ex}{\includegraphics[height=2.3ex]{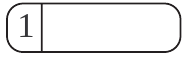}} \; \dot{\cup} \;
   \psi_{2} \cdot \raisebox{-0.5ex}{\includegraphics[height=2.3ex]{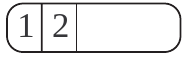}} \; \dot{\cup} \; \cdots \; \dot{\cup} \;
   \psi_{b-1} \cdot \raisebox{-0.5ex}{\includegraphics[height=2.3ex]{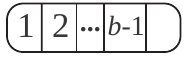}} \; \dot{\cup} \notag \\
   & \quad \varphi_{0} \cdot \raisebox{-0.5ex}{\includegraphics[height=2.3ex]{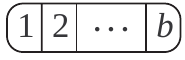}} \; \dot{\cup} \;
   \varphi_{1} \cdot \raisebox{-0.5ex}{\includegraphics[height=2.3ex]{bucket4.pdf}} \times \mathcal{T} \; \dot{\cup} \;
   \varphi_{2} \cdot \raisebox{-0.5ex}{\includegraphics[height=2.3ex]{bucket4.pdf}} \times \mathcal{T} \ast \mathcal{T} 
	\label{eqn:describ1} \\
   & = \psi_{1} \cdot \raisebox{-0.5ex}{\includegraphics[height=2.3ex]{bucket1.pdf}} \; \dot{\cup} \;
   \psi_{2} \cdot \raisebox{-0.5ex}{\includegraphics[height=2.3ex]{bucket2.pdf}} \; \dot{\cup} \; \cdots \; \dot{\cup} \;
   \psi_{b-1} \cdot \raisebox{-0.5ex}{\includegraphics[height=2.3ex]{bucket3.pdf}} \; \dot{\cup} \;
   \raisebox{-0.5ex}{\includegraphics[height=2.3ex]{bucket4.pdf}} \times \varphi\big(\mathcal{T}\big), \notag
\end{align}
where $\raisebox{-0.5ex}{\includegraphics[height=2.3ex]{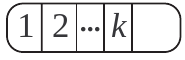}}$ denotes a bucket of capacity $k$ labelled by $1, 2, \dots, k$,
$\times$ the Cartesian product, $\ast$ the partition product for labelled
objects, and $\varphi(\mathcal{T})$ the substituted structure. On the other hand, we can also use standard notation~\cite{FlaSed} for formal specifications of combinatorial structures. Let $\mathcal{Z}$ denote the \emph{atomic class} (i.e., a single (uni)labelled node), $\mathcal{A}^{\Box} \ast \mathcal{B}$ the \emph{boxed product} (i.e., the smallest label is constrained to lie in the $\mathcal{A}$ component) of the combinatorial classes $\mathcal{A}$ and $\mathcal{B}$. Then, 
\begin{equation}
\begin{split}
\label{eqn:describ2}
 \mathcal{T} & = \psi_1\cdot \mathcal{Z}^{\Box} + \psi_2\cdot (\mathcal{Z}^{\Box})^2+ \dots
+ \psi_{b-1}\cdot (\mathcal{Z}^{\Box})^{b-1} + (\mathcal{Z}^{\Box})^{b}\ast \varphi(\mathcal{T})\\
 & = \sum_{k=1}^{b-1}\psi_{k}\cdot (\mathcal{Z}^{\Box})^{k} + (\mathcal{Z}^{\Box})^{b}\ast \varphi(\mathcal{T}).
\end{split}
\end{equation}
Here the meaning of $(\mathcal{Z}^{\Box})^{k} \ast \mathcal{B}$ is $\mathcal{Z}^{\Box} \ast \left(\mathcal{Z}^{\Box} \ast \left(\cdots \ast \left(\mathcal{Z}^{\Box} \ast \mathcal{B}\right)\right)\right)$, with $k$ occurrences of $\mathcal{Z}^{\Box}$.

\medskip

Using above formal descriptions~\eqref{eqn:describ1} or \eqref{eqn:describ2}, one can show that the exponential generating function
$T(z) := \sum_{n \ge 1} T_{n} \frac{z^{n}}{n!}$ of the total weights $T_{n}$ is characterized by the following result.
\begin{prop}[\cite{BucketPanKu2009}]\label{prop:DEQ_BIC}
The exponential generating function $T(z)$ of the total weights $T_{n}$ of bucket increasing trees with bucket-weight sequence $(\psi_{k})_{1\le k\le b-1}$ and degree-weight generating function $\varphi(t)$ satisfies 
an ordinary differential equation of order $b$:
\begin{align}
   \label{eqna1}
   \frac{d^{b}}{d z^{b}} T(z) & = \varphi\big(T(z)\big), 
\end{align}
with initial conditions
\[
   T(0)=0, \qquad T^{(k)}(0) = \psi_{k}, \quad \text{for} \enspace 1 \le k \le b-1.
\]
\end{prop}

\begin{remark}[Differential equation for $T'(z)$]
Assume that $\varphi'(t)=g\big(\varphi(t)\big)$ for some function $g(z)$. Then, the derivative $F(z)=T'(z)$
satisfies the differential equation
\[
   F^{(b)}(z)  - F(z)\cdot g\Big(F^{(b-1)}(z)\Big)=0.
\]
Such differential equations occur in applications in the context of the Prandtl-Blasius flow 
and are, somewhat curiously, related to bucket increasing trees~\cite{KubPan2016}.
\end{remark}

\begin{example}[Bucket ordered increasing trees]
A basic example of bucket increasing trees are bucket ordered increasing trees with $\psi_k=1$, $1 \le k \le b-1$ and $\varphi_k=1$, $k\ge 0$, i.e., $\varphi(t) = \frac{1}{1-t}$. They are used as the fundamental underlying tree family for all the other weighted bucket increasing trees and can be described combinatorially by the sequence operator,
\[
 \mathcal{T}= \sum_{k=1}^{b-1}(\mathcal{Z}^{\Box})^{k} + (\mathcal{Z}^{\Box})^{b}\ast \textsc{Seq}(\mathcal{T}),
\]
and the exponential generating function $T(z)$ satisfies
\[
 \frac{d^{b}}{d z^{b}} T(z) = \frac1{1-T(z)},
\qquad T(0)=0,\quad T^{(k)}(0)=1,\, 1\le k\le b-1.
\]
For $b=1$ such tree families are often called plane-oriented recursive trees and the total weights $T_{n}$, which are here simply the total numbers of such trees, are given by $T_n=(2n-3)!!$ (see \href{https://oeis.org/A001147}{A001147}). For $b=2$, an asymptotic expansion of the total numbers $(T_n)_{n\ge 1}=(1,1,1,3,13,77,\dots)$ (see \href{https://oeis.org/A032035}{A032035} and Figure~\ref{fig:bucket_ordered}) of such trees has been obtained by Bodini et al.~\cite{Hwang2016} in the context of increasing diamonds.

\begin{figure}%
\includegraphics[scale=0.4]{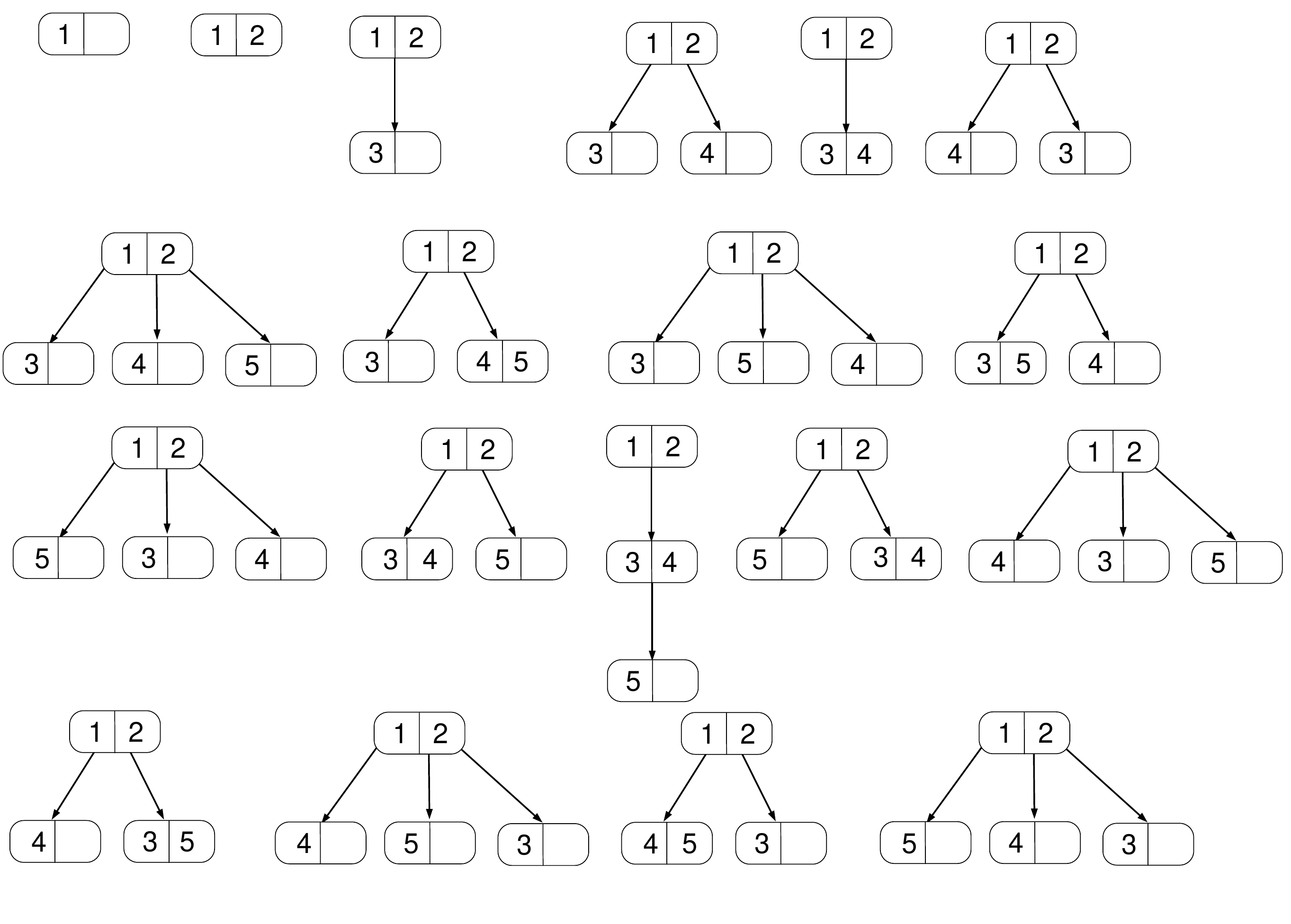}
\caption{Bucket ordered increasing trees of sizes one up to five.}%
\label{fig:bucket_ordered}
\end{figure}

\end{example}

\smallskip

Given a certain degree-weight sequence $(\varphi_{k})_{k \ge 0}$ and a bucket-weight sequence \newline
$(\psi_{k})_{1\le k\le b-1}$
specifying a family $\mathcal{T}$ of bucket increasing trees. We obtain random (ordered) bucket increasing trees $\mathcal{T}_n$ of size $n$, when assuming that each increasingly labelled bucket ordered tree $\tilde{T}\in\mathcal{T}_n$ of size $n$ is chosen with a probability proportional to its weight $w(\tilde{T})$:
\[ 
\P\{\tilde{T}\}=\frac{w(\tilde{T})}{T_n}=\frac{\displaystyle{\Big(\prod_{\substack{v\in V(\tilde{T}): \\ c(v)=b}}\varphi_{\deg(v)}\Big)\Big(\prod_{\substack{v\in V(\tilde{T}): \\ c(v)<b}}\psi_{c(v)}\Big)  }}{T_n}.
\]

For the presentation of the combinatorial model we choose what we consider the most natural model, i.e., starting always with a single tree of size one, similar to the case $b=1$ of ordinary increasing trees~\cite{PanPro2007}. However, sometimes it may be beneficial to alternatively consider different sequences corresponding, for example, to coloured trees. 
Thus, we note that the random bucket increasing trees specified by bucket-weights $\psi_k$ and degree-weights $\varphi_k$ are only unique up to scaling, as one can scale the weight sequences in the following way. 

\begin{lemma}[Random bucket increasing trees and scaling of weight sequences]
\label{LemScaling}
Given two scaling parameters $a,s>0$ and two pairs of degree-weight and bucket-weight sequences $(\hat{\varphi}_{k})_{k \ge 0}$, $(\hat{\psi}_{k})_{1\le k\le b-1}$ and $(\varphi_{k})_{k \ge 0}$,  $(\psi_{k})_{1\le k\le b-1}$, respectively, related by
\[
\hat{\varphi}_{k} = a^b \cdot s^{k-1} \cdot \varphi_{k}, \quad \text{for $k \ge 0$}, \qquad \text{and} \qquad
\hat{\psi}_{k} = a^k \cdot s^{-1} \cdot \psi_{k}, \quad \text{for $1 \le k \le b-1$},
\]
or equivalently, $\hat{\varphi}(t) = a^b s^{-1} \varphi(s t)$ for the corresponding degree-weight generating functions.

Then, both pairs of weight sequences lead to the same distribution of random bucket increasing trees.
\end{lemma}

\begin{proof}
We consider the weights $\hat{w}(\tilde{T})$ and $w(\tilde{T})$ of a tree $\tilde{T}$ with respect to $(\hat{\varphi}_{k})_{k \ge 0}$, $(\hat{\psi}_{k})_{1\le k\le b-1}$ and $(\varphi_{k})_{k \ge 0}$,  $(\psi_{k})_{1\le k\le b-1}$.
We have 
\begin{equation}
\begin{split}
\label{eqn:Scaling1}
   \hat{w}(\tilde{T}) &= \bigg(\prod_{\substack{v\in V(\tilde{T}): \\ c(v)=b}}\hat{\varphi}_{\deg(v)}\bigg)\bigg(\prod_{\substack{v\in V(\tilde{T}): \\ c(v)<b}}\hat{\psi}_{c(v)}\bigg)\\
	&= \bigg(\prod_{\substack{v\in V(\tilde{T}): \\ c(v)=b}} a^b \cdot s^{\deg(v)-1} \cdot \varphi_{\deg(v)}\bigg)\bigg(\prod_{\substack{v\in V(\tilde{T}): \\ c(v)<b}}a^{c(v)} \cdot s^{-1} \cdot \psi_{c(v)}\bigg)\\
	&=\bigg(\prod_{v\in V(\tilde{T})}a^{c(v)}s^{\deg(v)-1} \bigg)
	\bigg(\prod_{\substack{v\in V(\tilde{T}): \\ c(v)=b}}\varphi_{\deg(v)}\bigg)\bigg(\prod_{\substack{v\in V(\tilde{T}): \\ c(v)<b}}\psi_{c(v)}\bigg).
\end{split}
\end{equation}
The latter two products directly give the weight $w(\tilde{T})$. In order to simplify the first product, we use the properties
\begin{equation*}
   |\tilde{T}|= \sum_{v \in V(\tilde{T})} c(v), \quad \sum_{v\in V(\tilde{T})}(\deg(v)-1)=|V(\tilde{T})|-1 - |V(\tilde{T})| =-1.
	\end{equation*}
This implies that
\[
	\hat{w}(\tilde{T})=a^{\sum_{v\in V(\tilde{T})}c(v)} \cdot s^{\sum_{v\in V(\tilde{T})} \big(\deg(v)-1\big)} \cdot w(\tilde{T})
	= a^{|\tilde{T}|} s^{-1} w(\tilde{T}).
\]
Thus, when changing $(\varphi_{k})_{k\ge 0}$ to $(\hat{\varphi}_{k})_{k\ge 0}$ and also the bucket-weights correspondingly, the weight of any tree $\tilde{T}$ of size $n$ will be multiplied by the same factor $a^{n} s^{-1}$, which will affect the weight of all trees of size $n$ also by the same factor, leading to the same probability $\P\{\tilde{T}\}$ for both degree-weight and bucket-weight sequences.
\end{proof}

\subsection{Tree evolution processes and combinatorial models}

We collect the three growth processes generating bucket increasing trees~\cite{BucketPanKuAccepted,MahSmy1995} and the corresponding combinatorial descriptions from~\cite{BucketPanKu2009,BucketPanKuAccepted}. Note that here and throughout this work the capacities $c(v)=c_n(v)$ and the out-degree $\grad(v)=\grad_n(v)$ of a node $v$ in a tree $T$ are always dependent on the size $|T|=n$. We also mention that, from this point on, $T\in\mathcal{T}$ denotes a bucket increasing tree and not, as previously used, its unlabelled bucket ordered counterpart.

\begin{defi}[Bucket recursive trees]
\label{def0}
For the tree evolution process, we start with a single bucket as root node containing only label $1$. 
Given a tree $T$ of size $n\ge 1$. Let $p(v)=\P\{n+1 <_t v \mid c(v)\}$ denote the probability that node $v\in T$ attracts label $n+1$ conditioned on its capacity $c(v)$.
The family of random \emph{bucket recursive trees} is generated according to the probabilities 
\[
   p(v)  = \frac{c(v)}{n},
\] 
with capacity $1\le c(v)\le  b$, thus independent of the out-degree $\grad(v)\ge 0$ of node $v$.

\smallskip

A combinatorial model of bucket recursive trees is determined by the degree-weight and bucket-weight sequences
\begin{equation*}
   \varphi_{k} = \frac{(b-1)! \, b^{k}}{k!}, \quad \text{for} \enspace k \ge 0, \qquad \psi_{k} = (k-1)!, \quad \text{for} \enspace 1 \le k \le b-1,
\end{equation*}
such that $\varphi(t)=\sum_{k\ge 0}\varphi_k t^k=(b-1)! \cdot \exp(b t)$.
\end{defi}

\smallskip

\begin{defi}[\Dit]
\label{def1}
For the tree evolution process, we start, case $n=1$, with a single bucket as root node containing only label $1$. 
Given a tree $T$ of size $n\ge 1$. Again, let $p(v)$ denote the probability that node $v\in T$ attracts label $n+1$ in a bucket increasing tree of size $n\in\N$. 

\smallskip

The family of random \emph{(b,d)-ary increasing trees}, with $d\in\Q$ such that $(d-1)b\in\N_0$, is generated according to the probabilities
\[
   p(v)  =\displaystyle{ \frac{(d-1)c(v)+1-\grad(v)}{(d-1)n+1}}, 
\]
with $1\le c(v)\le  b$ and $\grad(v)\ge 0$.

\smallskip

A combinatorial model of $(b,d)$-ary increasing trees is determined by the degree-weight and bucket-weight sequences
\begin{equation*}
\begin{split}
   \varphi_{k} &= (b-1)!(d-1)^{b-1}\binom{b-1+\frac{1}{d-1}}{b-1}\binom{b(d-1)+1}{k}, \quad \text{for} \enspace k \ge 0,\\
   \psi_{k} &= (k-1)!(d-1)^{k-1}\binom{k-1+\frac{1}{d-1}}{k-1}, \quad \text{for} \enspace 1 \le k \le b-1,
\end{split}
\end{equation*}
such that $\varphi(t)= (b-1)!(d-1)^{b-1}\binom{b-1+\frac{1}{d-1}}{b-1}(1+t)^{b(d-1)+1}$.
\end{defi}

\smallskip

\begin{defi}[\Port]
\label{def2}
For the tree evolution process, we start, case $n=1$, with a single bucket as root node containing only label $1$. 
Given a tree $T$ of size $n\ge 1$. Let $p(v)$ denote the probability that node $v\in T$ attracts label $n+1$ in a bucket increasing tree of size $n\in\N$. 

The family of random \emph{(b,$\alpha$)-plane oriented recursive trees}, with $\alpha>0$, is generated according to the probabilities
\[
   p(v)  =  \displaystyle{ \frac{\grad(v)+(\alpha+1)c(v)-1}{(\alpha+1)n-1}},
\]
with $1\le c(v)\le  b$ and $\grad(v)\ge 0$.

\smallskip

A combinatorial model of $(b,\alpha)$-plane oriented recursive trees is determined by the degree-weight and bucket-weight sequences
\begin{equation*}
\begin{split}
   \varphi_{k} &=(b-1)!(\alpha+1)^{b-1}\binom{b-1-\frac{1}{\alpha+1}}{b-1}\binom{(\alpha+1)b-2+k}{k}, \quad \text{for} \enspace k \ge 0,\\
   \psi_{k} &= (k-1)!(\alpha+1)^{k-1}\binom{k-1-\frac{1}{\alpha+1}}{k-1}, \quad \text{for} \enspace 1 \le k \le b-1,
\end{split}
\end{equation*}
such that $\varphi(t)=\frac{(b-1)!(\alpha+1)^{b-1}\binom{b-1-\frac{1}{\alpha+1}}{b-1}}{(1-t)^{(\alpha+1)b-1}}$.
\end{defi}

We will see in Theorem~\ref{the:evo} that only half of the previously stated definitions are required:
the tree evolution processes determine the bucket-weight sequences as well as the degree-weight sequences and vice versa. 

\smallskip

From the definitions above and the differential equation~\eqref{eqna1} one directly obtains the enumerative results of~\cite{BucketPanKu2009,BucketPanKuAccepted}, which we restate below. These results are particularly of interest, as non-linear differential equations of order greater or equal three occurring in enumeration problems are extremely seldom to be solved in closed form; see for example variations of the differential equations in connection with the Blasius-type tree family~\cite{KubPan2016}, as well as~\cite{Hwang2016}.

\begin{prop}[\cite{BucketPanKu2009,BucketPanKuAccepted}]
\label{PropEnum}
The exponential generating functions $T(z)$ and the total weights $T_n$ of tree families generated according to Definition~\ref{def0},\ref{def1} and \ref{def2}, respectively, are given as follows.
\begin{itemize}
	\item Bucket recursive trees: 
	\begin{equation*}
T(z)=\log\big(\frac{1}{1-z}\big),\quad T_n=(n-1)!.
\end{equation*}
 \item \Dit: 
\[
T(z)=\frac{1}{(1-(d-1)z)^{\frac1{d-1}}}-1,\quad T_n=(n-1)!(d-1)^{n-1}\binom{n-1+\frac{1}{d-1}}{n-1}.
\] 
\item \Port:
\[
T(z)=1-(1-(\alpha+1)z)^{\frac1{\alpha+1}},\quad T_n=(n-1)!(\alpha+1)^{n-1}\binom{n-1-\frac{1}{\alpha+1}}{n-1}.
\]
\end{itemize}
\end{prop}
\begin{remark}[Weight-preserving families]
We note in passing (which is apparent from the formul{\ae} given) that the tree families in Definitions~\ref{def0},\ref{def1} and \ref{def2} preserve the total weights $T_n$ of the corresponding ordinary increasing tree families (bucket size $b=1$), as the total weights are independent of the bucket size $b$; see~\cite{BucketPanKuAccepted} for details.
\end{remark}

\section{Bucket increasing trees and tree Evolution processes}
In the following theorem we state the main result of this work, namely six different equivalent properties, characterizing 
bucket increasing trees generated by a tree evolution process.
\begin{theorem}
\label{the:evo}
The following properties of families of bucket increasing trees $\mathcal{T}$ are equivalent:

\begin{enumerate}[label=(\roman*)]
\item \mt{Combinatorial model}: The family $\mathcal{T}$ can be modelled combinatorially by bucket-weights $\psi_k$ and degree-weights 
$\varphi_k$ as given in Definition~\ref{def0},\ref{def1} and \ref{def2}, respectively, up to scaling. 

\smallskip

\item \mt{Affine linear ratio}: The total weights $T_n>0$ of trees of size $n$ of the family $\mathcal{T}$ satisfy for all $n\in\N$ the equation
\begin{equation}
\frac{T_{n+1}}{T_n}= c_1\cdot n + c_2,
\end{equation}
with fixed real constants $c_1$, $c_2$.

\smallskip

\item \mt{Tree evolution process}: The family $\mathcal{T}$ is generated according to the tree evolution processes as described in Definition~\ref{def0},\ref{def1} and \ref{def2}, respectively.

\smallskip

\item \mt{Probabilistic growth rule}: The family $\mathcal{T}$ can be constructed via an insertion process (resp.\ a probabilistic growth
rule), i.e., for every tree $T'$ of size $n$ with vertices  $v_j$, $1\le j\le |V(T')|$, there exist probabilities
$p_{T'}(v_j)$ such that when starting with a random tree $T'$ of size $n$, choosing
a vertex $v_j$ in $T'$ according to the probabilities $p_{T'}(v_j)$ and attaching label $n+1$ to it, we obtain
a random increasing tree $T$ of size $n+1$.

\smallskip

\item \mt{Preservation of randomness}: Starting with a random increasing tree $T\in\mathcal{T}$ of size $n\ge j$ and removing all labels larger than $j$ we obtain a random bucket increasing tree $T_0$ of size $j$ of the family $\mathcal{T}$. 

\smallskip

\item \mt{Balance}: Given a tree $T \in \mathcal{T}$, let  $m_{k}=m_k(T)= |\{u \in T : c(u)=k<b\}|$ be the number of unsaturated nodes of $T$ with capacity $k<b$ and $n_{k} =n_k(T)= |\{u \in T : c(u)=b \enspace \text{and} \enspace \grad(u)=k\}|$ be the number of saturated nodes of $T$ with out-degree $k \ge 0$. For all trees $T\in \mathcal{T}$ with $|T|=n$ the combinatorial model of $\mathcal{T}$ satisfies 
the balance condition
\begin{equation}
C_n=\sum_{k=1}^{b-1} m_{k} \frac{\psi_{k+1}}{\psi_{k}} + \sum_{k \ge 0} n_{k} (k+1) \psi_{1} \frac{\varphi_{k+1}}{\varphi_{k}},
\label{eq:balance}
\end{equation}
with $C_n$ being independent of the particular tree $T$.
\end{enumerate}
\end{theorem}

\begin{remark}[Connectivity]
The quantity $C_n$ in~\eqref{eq:balance} is essentially the total connectivity and given by the denominators $q_{n} =n$, $q_{n} = (d-1)n+1$ and $q_{n}=(\alpha+1)n-1$, respectively, of the probabilities $p(v)$ in Definition~\ref{def0}, \ref{def1} and \ref{def2}, respectively.
\end{remark}

\begin{remark}[Label-dependent probabilistic growth rules]
We emphasize that in the probabilistic growth rule the indices $j$ of the vertices $v_j$ \emph{do not}
refer to the labels, but simply to the different vertices (or buckets) in the tree $T'$. It is possible to obtain different families of random (bucket) increasing trees without the random tree model. For example, instead of taking the out-degree of vertices into account, we can create trees in a step-by-step fashion using a label-dependent growth rule. Given a weight sequence $(w_k)_{k\ge 0}$, such that $w_1>0$. Let $n\ge 1$. Starting with a tree of size $n$ and bucket size $b=1$, we attach label $n+1$ to the vertex labeled $k$ with probability
\[
\P\{n+1<_c k\}=\frac{w_k}{W_n},
\]
with total weight $W_n=\sum_{k=1}^{n}w_k$. Such trees are known in the literature (see, e.g., \cite{BorVat2006}) as weighted recursive trees, as for constant $w_k=c$, $k\ge 1$, this reduces to ordinary recursive trees.
\end{remark}

\begin{proof}
We prove Theorem~\ref{the:evo} by providing several implications. 
Throughout, we set for convenience $\psi_b=\varphi_0$. 
As a quick overview, we show that $(i)\impli(ii)$, $(ii)\impli(i)$, $(i)\impli (vi)$, $(i)\impli(iii)$, $(iii)\impli(iv)$,
$(iv)\impli(v)$, $(iv)\impli(vi)$, $(v)\impli(vi)$ $(vi)\impli(i)$.

\smallskip

%

\textbf{Case (i)\impli(ii).} Due to the demand $T_{n} > 0$, for all $n \ge 1$, we get the restrictions:
   \begin{equation*}
	c_{1} \ge 0\quad \text{ and }\quad c_{2} > - c_{1},
	\end{equation*}
	since otherwise, there would exist $n \ge 1$ such that
   \[
	\frac{T_{n+1}}{T_{n}} = c_{1} n + c_{2} < 0.
	\]
	The bucket-weights $\psi_k$ are determined by the equation
	\[
	\psi_k=T_k,\quad 1\le k\le b-1. 
	\]

  Moreover, we have $\varphi_0=T_b$. Now we consider the subcase $c_{1} \neq 0$ and $c_{2} \neq 0$ and get for $T_{n}$:
   \begin{align*}
      T_{n} & = T_{1} \prod_{k=1}^{n-1} (c_{1} k + c_{2}) = \psi_{1} c_{1}^{n-1}
      \prod_{k=1}^{n-1} \big(\frac{c_{2}}{c_{1}} + k \big)
      = \frac{\psi_{1} c_{1}^{n}}{c_{2}}
      \big(\frac{c_{2}}{c_{1}} + n - 1 \big)^{\underline{n}}
      \\
      & =\frac{\psi_{1} c_{1}^{n} n!}{c_{2}} \binom{\frac{c_{2}}{c_{1}} + n-1}{n} = \frac{\psi_{1} (- c_{1})^{n} n!}{c_{2}} \binom{-\frac{c_{2}}{c_{1}}}{n},
   \end{align*}
   and further
   \begin{equation}
      \label{eqni1}
      T(z) = \sum_{n \ge 1} T_{n} \frac{z^{n}}{n!} = \frac{\psi_{1}}{c_{2}}
      \sum_{n \ge 1} \binom{-\frac{c_{2}}{c_{1}}}{n} (- c_{1} z)^{n}
      = \frac{\psi_{1}}{c_{2}} \Big(\frac{1}{(1-c_{1} z)^{\frac{c_{2}}{c_{1}}}} -1 \Big).
   \end{equation}
	This implies that
	\begin{equation}
      \label{eq:zOfT}
	z= z(T)=\frac{\Big(\frac{c_2}{\psi_1}\cdot T+1\Big)^{-\Cc}-1}{-c_1}.
	\end{equation}
  
	In order to decide, which values of $c_{1}$, $c_{2}$ are indeed possible choices, we have to
   compute the corresponding degree-weight generating functions. 
	Then, we check, whether they are
   admissible, such that $\varphi_{k} \ge 0$ for all $k \ge 0$, and non-degenerate, such that there exists a $k\ge 2$ with $\varphi_k>0$.
  
	\smallskip
	
	We differentiate $T(z)$, given in \eqref{eqni1}, $b$ times with respect to $z$ and obtain
   \begin{equation*}
      T^{(b)}(z) = \frac{\psi_{1}}{c_2}\cdot c_1^b\cdot \binom{b-1+\cC}{b}\cdot b!\cdot \frac{1}{\big(1-c_1z\big)^{\cC+b}}.
   \end{equation*}
  Since $\varphi(T) =T^{(b)}\big(z(T)\big)$, we get by using \eqref{eq:zOfT} the intermediate result	
	\begin{equation}
      \label{eqnd2}
      \varphi(T) = \frac{\psi_{1}}{c_2}\cdot c_1^b\cdot \binom{b-1+\cC}{b}\cdot b!\cdot \Big(1 + \frac{c_{2}}{\psi_{1}} T\Big)
			^{1+b\Cc}.
   \end{equation}
  Extracting coefficients, $\varphi_k=[T^k]\varphi(T)$, gives
   \begin{equation*}
	\begin{split}
      \varphi_k&=	\frac{\psi_{1}}{c_2}\cdot c_1^b\cdot \binom{b-1+\cC}{b}\cdot b!\cdot \binom{b\Cc+1}{k}\cdot \frac{c_2^k}{\psi_1^k}\\
		&	=	c_1^{b-1}\cdot \binom{b-1+\cC}{b-1}\cdot (b-1)!\cdot \binom{b\Cc+1}{k}\cdot \frac{c_2^k}{\psi_1^{k-1}}.
  \end{split} 
	\end{equation*}
   We can now check, whether the conditions
   $\varphi_{k} \ge 0$, for all $k \ge 0$, with $\varphi_{0} > 0$,
   for the degree-weight sequence are satisfied. To do this we distinguish several cases.

  \begin{itemize}
   \item First we consider the case $c_{2} > 0$: if $1 + b\Cc \not\in \mathbb{N}$, then it follows that there exists a $k \in \mathbb{N}$ such that $\binom{b\Cc+1}{k} < 0$. Since $c_{1} > 0$, this implies that $\varphi_{k} < 0$. Therefore, this case is not admissible.
   On the other hand, if $1 + b\Cc= : D \in \mathbb{N}$,
   then it follows that $\binom{D}{k} = 0$, for all $k > D$, and thus
   that $\varphi_{k} > 0$, for all $0 \le k \le D$, and $\varphi_{k} = 0$, for all $k > D$.
   Such degree-weight generating functions are admissible and are covered by the \Dit. 
	There, according to Lemma~\ref{LemScaling}, we might make the specific choice $\Cc=d-1$ and set $c_2=\psi_1=1$, starting with a single tree of size one.
  
	\smallskip
	
	\item Next, we consider the case $c_{2} < 0$: since $c_{1} + c_{2} > 0$, it follows
   that $\frac{c_{1}}{c_{2}} < -1$. 
	We set $\Cc=-(\alpha+1)<-1$. The specific choice $c_1=\alpha+1$, $-c_2=\psi_1=1$, then leads to
   \begin{equation*}
      \varphi_{k} = (\alpha+1)^{b-1}\binom{b-1-\frac{1}{\alpha+1}}{b-1}\cdot(b-1)!\cdot \binom{b(\alpha+1)+k-2}{k},
   \end{equation*}
   for all $k \ge 0$. Again, according to Lemma~\ref{LemScaling}, other choices yield equivalent random tree models. Apparently, such degree-weight generating functions are also admissible and are covered by \Port.
	
	\smallskip
	
	\item Eventually we consider the case $c_{2} = 0$, which gives
   \begin{equation*}
      T_{n} = T_{1} \prod_{k=1}^{n-1} (c_{1} k) = \psi_{1} c_{1}^{n-1} (n-1)!,
   \end{equation*}
   and
   \begin{equation}
      \label{eqnd5}
      T(z) = \sum_{n \ge 1} T_{n} \frac{z^{n}}{n!} = \frac{\psi_{1}}{c_{1}}
      \sum_{n \ge 1} \frac{(c_{1} z)^{n}}{n} = \frac{\psi_{1}}{c_{1}}
      \log\big(\frac{1}{1-c_{1}z}\big).
   \end{equation}
Equation~\eqref{eqnd5} gives
   \begin{equation}
      \label{eqnd6}
      T^{(b)}(z) = \frac{\psi_1}{c_1}\cdot c_1^b\cdot\frac{(b-1)!}{(1-c_{1} z)^b}.
   \end{equation}
   We obtain further
   \begin{equation}
      \label{eqnd7}
			\varphi(T)=T^{(b)}\big(z(T)\big)=
			\frac{\psi_1}{c_1}\cdot c_1^b (b-1)! \cdot \exp\Big(T\cdot b\frac{c_1}{\psi_0}\Big)
   \end{equation}
   and also
   \begin{equation}
	\label{eqnd8}
      \varphi_k=c_1^b\cdot\frac{c_1^{k-1}}{\psi_1^{k-1}}(b-1)!\frac{b^k}{k!},\quad k\ge 0.
			\end{equation}
   Since $c_{1} > 0$, we obtain from \eqref{eqnd8} that $\varphi_{k} > 0$, for all $k \ge 0$,
   and thus that all such degree-weight generating functions are admissible.
   They are covered by bucket recursive trees, where we made the special choice $c_1=\psi_1=1$, 
	such that we start again with a single tree of size one. Thus, according to Lemma~\ref{LemScaling}, without loss of generality 
	we have again characterized all possible weight sequences.

\item The remaining case $c_{1} = 0$ and $c_{2} > 0$ leads to
   $T_{n} = \psi_{1} c_{2}^{n-1}$ and to
   \begin{equation}
      \label{eqnd9}
      T(z) = \sum_{n \ge 1} T_{n} \frac{z^{n}}{n!} = \frac{\psi_{1}}{c_{2}}
      \sum_{n \ge 1} \frac{(c_{2} z)^{n}}{n!} = \frac{\psi_{1}}{c_{2}}
      \big(e^{c_{2} z} -1\big).
   \end{equation}
   Since \eqref{eqnd9} gives
   \begin{equation}
      T^{(b)}(z) = \psi_{1} c_2^{b-1}e^{c_{2} z},
   \end{equation}
   this leads to
   \begin{equation}
      \varphi(T) =  T^{(b)}\big(z(T)\big)= \psi_1c_2^{b-1} + c_2^b\cdot T.
   \end{equation}
   Here all trees are chains and this degenerate case is excluded from our further considerations
   due to the demand that there exists a $k \ge 2$ with $\varphi_{k} > 0$.
\end{itemize}

\medskip

\textbf{Case (ii)\impli (i).} 
We use Proposition~\ref{prop:DEQ_BIC} and \ref{PropEnum}: by the combinatorial specification of bucket increasing trees, the exponential generating function $T(z)$ of the total weights $T_{n}$ satisfies the non-linear differential equation given in~\eqref{eqna1}:
\[
   \frac{d^{b}}{d z^{b}} T(z)  = \varphi\big(T(z)\big), \; T(0)=0, \qquad T^{(k)}(0) = \psi_{k}, \quad \text{for} \enspace 1 \le k \le b-1.
\]
One can readily check that the generating functions stated in Proposition~\ref{PropEnum} solve the corresponding differential equations with weights specified in Definition~\ref{def0},\ref{def1} and \ref{def2}, respectively. 
Thus, the corresponding degree-weights and bucket-weights lead to the formulas for the total weights $T_n$ stated in Proposition~\ref{PropEnum} and it can be checked easily that they satisfy the asserted affine linear ratio.
Namely, for bucket recursive trees we get
\[
\frac{T_{n+1}}{T_n}=n,\quad c_1=1,\quad c_2=0.
\]
For \Dit\, we obtain 
\[
\frac{T_{n+1}}{T_n}=(d-1)n+1\cdot ,\quad c_1=(d-1),\quad c_2=1.
\]
Finally, for \Port\, we get
\[
\frac{T_{n+1}}{T_n}=(\alpha+1)n-1\cdot ,\quad c_1=(\alpha+1),\quad c_2=-1.
\]

\medskip

\textbf{Case: (i)\impli (vi).}
We will use the properties
\begin{equation}
\label{BINCeqn001}
   |\tilde{T}|= \sum_{u \in \tilde{T}} c(u) = \sum_{k=1}^{b-1} k m_{k} + b \sum_{k \ge 0} n_{k},
\end{equation}
where as before $m_{k}=m_k(\tilde{T})= |\{u \in \tilde{T} : c(u)=k<b\}|$ denotes the number of unsaturated nodes of $\tilde{T}$ with capacity $k<b$ and $n_{k} =n_k(\tilde{T})= |\{u \in \tilde{T} : c(u)=b \enspace \text{and} \enspace \grad(u)=k\}|$ the number of saturated nodes of $\tilde{T}$ with out-degree $k \ge 0$. We also require the relation
\begin{equation}
\label{BINCeqn002}
 1 = \sum_{k=1}^{b-1} m_{k} - \sum_{k \ge 0} (k-1)n_{k},
\end{equation}
which follows as the difference between the node-sum and edge-sum equation for the tree $\tilde{T}$:
\begin{equation*}
\text{$\#$ nodes} = \sum_{k=1}^{b-1} m_{k} + \sum_{k \ge 0} n_{k}, \qquad \text{$\#$ edges} = \text{$\#$ nodes} - 1 = \sum_{k \ge 0} k n_{k}.
\end{equation*}

Using equations~\eqref{BINCeqn001} and~\eqref{BINCeqn002}, we can easily verify that the balance condition is satisfied for the three combinatorial families of increasing trees.
\begin{itemize}
\item For bucket recursive trees with weights $\psi_{k} = (k-1)!$ and $\varphi_{k} = \frac{(b-1)! b^{k}}{k!}$ we obtain
\begin{equation*}
   \sum_{k=1}^{b-1} k m_{k} + \psi_{1} \sum_{k \ge 0} n_{k} (k+1) \frac{b}{k+1}
   = \sum_{k=1}^{b-1} k m_{k} + b \sum_{k \ge 0} n_{k} = n=: C_{n}
\end{equation*}

\item Next, we turn our attention to the family of \Dit{} and its weight-sequences. Here we get
\begin{equation*}
\begin{split}
   &\sum_{k=1}^{b-1} (k(d-1)+1) m_{k} + \sum_{k \ge 0} n_{k} (b(d-1)+1-k)\\
   &\, = (d-1)\big(\sum_{k=1}^{b-1} k m_{k} + b \sum_{k \ge 0} n_{k}\big) + \sum_{k=1}^{b-1} m_{k} - \sum_{k \ge 0} (k-1)n_{k} =(d-1)n+1 =: C_{n}.
\end{split}
\end{equation*}

\item Finally, for \Port{} we obtain
\begin{equation*}
\begin{split}
   &\sum_{k=1}^{b-1} (k(\alpha+1)-1) m_{k} + \sum_{k \ge 0} n_{k} ((\alpha+1)b-1+k)\\
   &\,= (\alpha+1)\big(\sum_{k=1}^{b-1} k m_{k} + b \sum_{k \ge 0} n_{k}\big) - \big(\sum_{k=1}^{b-1} m_{k} - \sum_{k \ge 0} (k-1)n_{k}\big) =(\alpha+1)n-1 =: C_{n}.
\end{split}
\end{equation*}
\end{itemize}

\medskip
 
\textbf{Case: (i)\impli (iii).} To prove that the choices of sequences $(\varphi_k)_{k\in\N}$ and $(\psi_k)_{1\le k\le b-1}$ given in Definition~\ref{def0}, \ref{def1} and \ref{def2}, respectively, are actually models for bucket increasing trees generated according to the respective stochastic growth rules, we have to show that the corresponding combinatorial families $\mathcal{T}$ of bucket increasing trees have the same stochastic growth rules as the counterparts created probabilistically. 
Given an arbitrary bucket increasing tree $T \in \mathcal{T}$ of size $|T|=n$, 
then the probability that a new element $n+1$ is attracted by a node $v \in T$, with capacity $c(v)=k$ and out-degree $\grad(v)$,
has to coincide with the corresponding probability stated in Definitions~\ref{def0}-\ref{def2}.

We use now the notation $T \to T'$ to denote that $T'$ is obtained from $T$ with $|T|=n$ by incorporating element $n+1$, i.e., either by attaching element $n+1$ to a saturated node $v \in T$ at one of the $\grad(v)+1$ possible positions
(recall that bucket increasing trees are ordered trees by definition and thus the order of the subtrees is of relevance)
by creating a new bucket of capacity $1$ containing element $n+1$ or by adding element $n+1$ to an unsaturated node $v \in T$ by increasing
the capacity of $v$ by $1$. If we want to express that node $v \in T$ has attracted the element $n+1$ leading from $T$ to $T'$ we use
the notation $T \xrightarrow{v} T'$. If there exists a stochastic growth rule for a bucket increasing tree family $\mathcal{T}$,
then it must hold that, for a given tree $T \in \mathcal{T}$ of size $|T|=n$ and a given node $v \in T$, the probability $p_{T}(v)$, which gives
the probability that element $n+1$ is attracted by node $v \in T$, is given as follows:
\begin{equation}
   p_{T}(v) = \frac{\sum_{T' \in \mathcal{T} : T \xrightarrow{v} T'} w(T')}{\sum_{\tilde{T} \in \mathcal{T} : T \to \tilde{T}} w(\tilde{T})}
   = \frac{\sum_{T' \in \mathcal{T} : T \xrightarrow{v} T'} \frac{w(T')}{w(T)}}{\sum_{\tilde{T} \in \mathcal{T} : T \to \tilde{T}}
   \frac{w(\tilde{T})}{w(T)}}.
\end{equation}
The remaining task is to simplify the expression above into the form stated in Definitions~\ref{def0}-\ref{def2}. For a certain tree $\tilde{T}$ with $T \xrightarrow{u} \tilde{T}$ and $u \in T$, the quotient of the weight of the trees $\tilde{T}$ and $T$
is due to the definition of bucket increasing trees given as follows (recall that we set $\psi_{b} := \varphi_{0}$):
\begin{equation*}
   \frac{w(\tilde{T})}{w(T)} =
   \begin{cases}
      \psi_{1} \frac{\varphi_{k+1}}{\varphi_{k}}, & \quad \text{for} \enspace c(u)=b \quad \text{and} \quad \grad(u)=k, \\
      \frac{\psi_{k+1}}{\psi_{k}}, & \quad \text{for} \enspace c(u)=k < b.
   \end{cases}
\end{equation*}
Then it holds
\begin{equation*}
   \sum_{\tilde{T} \in \mathcal{T} : T \to \tilde{T}} \frac{w(\tilde{T})}{w(T)}
   = \sum_{k=1}^{b-1} m_{k} \frac{\psi_{k+1}}{\psi_{k}} + \sum_{k \ge 0} n_{k} (k+1) \psi_{1} \frac{\varphi_{k+1}}{\varphi_{k}},
\end{equation*}
where we use that there are $k+1$ possibilities of attaching a new node to a saturated node $u \in T$ with out-degree
$\grad(u)=k$. This is exactly the expression occurring in the balance condition already established in the previous argumentation, where we also computed the resulting values $C_{n}=n$, $C_{n} = (d-1)n+$ and $C_{n}=(\alpha+1)n-1$, respectively. Let us treat the different tree families separately.
\begin{itemize}
\item First, if one chooses the weights $\psi_{k} = (k-1)!$ and $\varphi_{k} = \frac{(b-1)! b^{k}}{k!}$ as in the family of bucket recursive trees, we obtain
\begin{equation*}
   \sum_{\tilde{T} \in \mathcal{T} : T \to \tilde{T}} \frac{w(\tilde{T})}{w(T)} = n.
\end{equation*}
Furthermore by choosing these weights $\varphi_{k}$ and $\psi_{k}$ we get
\begin{equation*}
   \sum_{T' \in \mathcal{T} : T \xrightarrow{v} T'} \frac{w(T')}{w(T)} =
   \begin{cases}
      (k+1) \psi_{1} \frac{\varphi_{k+1}}{\varphi_{k}} = b, & \quad \text{for} \enspace c(v)=b \enspace \text{and} \enspace \grad(v)=k,\\
      \frac{\psi_{k+1}}{\psi_{k}} = k, & \quad \text{for} \enspace c(v)=k < b,
   \end{cases}
\end{equation*}
and thus
\begin{equation*}
   \sum_{T' \in \mathcal{T} : T \xrightarrow{v} T'} \frac{w(T')}{w(T)} = k, \quad \text{for} \enspace c(v)=k, \quad 1 \le k \le b.
\end{equation*}
Therefore, we have shown that by choosing the weight sequences $\psi_{k} = (k-1)!$ and $\varphi_{k} = \frac{(b-1)! b^{k}}{k!}$
the probability $p_{T}(v)$ that in a bucket increasing tree $T$ of size $|T|=n$ the node $v$ with capacity $c(v)=k$ attracts element $n+1$
is always given by $\frac{k}{n}$, which coincides with the stochastic growth rule for bucket recursive trees.

\item Next, we turn our attention to the family of \Dit{} and its weight-sequences. We have 
\begin{equation*}
   \sum_{T' \in \mathcal{T} : T \xrightarrow{v} T'} \frac{w(T')}{w(T)} =
   \begin{cases}
			b(d-1)+1-k, & \text{for} \enspace c(v)=b \enspace \text{and} \enspace \grad(v)=k,\\
			k(d-1)+1, & \text{for} \enspace c(v)=k < b,
   \end{cases}
\end{equation*}
and we already know that
\begin{equation*}
\begin{split}
   \sum_{\tilde{T} \in \mathcal{T} : T \to \tilde{T}} \frac{w(\tilde{T})}{w(T)} &=(d-1)n+1.
\end{split}
\end{equation*}
Thus, with this choice of weight sequences $(\varphi_k)_k$ and $(\psi_k)_k$, 
the probability $p_{T}(v)$ that in a bucket increasing tree $T$ of size $|T|=n$ the node $v$ with capacity $c(v)=k$ attracts element $n+1$
coincides with the corresponding probability of the stochastic growth rule for \Dit{}.

\item For the family of \Port{} we obtain
\begin{equation*}
   \sum_{T' \in \mathcal{T} : T \xrightarrow{v} T'} \frac{w(T')}{w(T)} =
   \begin{cases}
			(\alpha+1)b-1+k, & \text{for} \enspace c(v)=b \enspace \text{and} \enspace \grad(v)=k, \\
			k(\alpha+1)-1, & \text{for} \enspace c(v)=k < b,
   \end{cases}
\end{equation*}
and we already gained that
\begin{equation*}
\begin{split}
   \sum_{\tilde{T} \in \mathcal{T} : T \to \tilde{T}} \frac{w(\tilde{T})}{w(T)} &=(\alpha+1)n-1.
\end{split}
\end{equation*}
Again, with this choice of weight sequences $(\varphi_k)_k$ and $(\psi_k)_k$, it follows that
the probability $p_{T}(v)$ that in a bucket increasing tree $T$ of size $|T|=n$ the node $v$ with capacity $c(v)=k$ attracts element $n+1$
coincides with the corresponding probability in the stochastic growth rule for \Port{}.
\end{itemize}

\medskip

\textbf{Case (iii)\impli (iv).} This is evidently true, as the tree evolution processes given in Definitions~\ref{def0}-\ref{def2} explicitly state the probabilities $p_{T'}(v)$ for all vertices $v\in T'$ of size $n$, and the resulting tree $T$ of size $n+1$ is again random, as the tree is created in a step-by-step fashion according to the tree evolution process.

\medskip

\textbf{Case (iv)\impli (v).} It is sufficient to show that, starting with a size-$n$ random tree and removing node $n$, the resulting tree of size $n-1$ is again random. But this is obviously true, as the tree evolution processes generate random trees in a step-by-step fashion and the underlying size-$(n-1)$ tree was random. 

\medskip

\textbf{Case (iv)\impli (vi).} We assume that, for every tree $T'\in\mathcal{T}$ of size $n$ with vertices $v'\in V(T')$, there exist such probabilities $p_{T'}(v')$; of course, $\sum_{v'\in V(T')}p_{T'}(v')=1$. Thus, starting with a random tree $T'$ of size $n$ and attaching the new label $n+1$ to one of the nodes $v\in V(T')$, either for $c(v)=b$ equally likely at any of the $\deg(v)+1$ different positions, or by inserting into nodes $v$ of capacity $c(v)<b$, leads then to a random tree $T$ of size $n+1$. 

\smallskip 

Given two trees $T',T''\in\mathcal{T}$, both of size $n$, with vertices $V(T')$ and $V(T'')$. Due to our assumption,
$\forall v \in V(T')\colon\ \exists \ p_{T'}(v)$ such that $\sum_{v\in V(T')}p_{T'}(v)=1$,
as well as $\forall v \in V(T'')\colon\ \exists \ p_{T''}(v)$ such that $\sum_{v\in V(T'')}p_{T''}(v)=1$. Attaching label $n+1$ to a fully saturated vertex $v\in V(T')$, $c(v)=b$, at any of the $\deg(v)+1$ positions, gives a tree $\tilde{T}'$ of size $n+1$ with weight
\[
w(\tilde{T}')=\psi_1\frac{\varphi_{\deg(v)+1}}{\varphi_{\deg(v)}}\cdot w(T').
\]
Likewise, attaching label $n+1$ to a vertex $v\in V(T')$ with $c(v)<b$ gives a tree $\tilde{T}'$ of size $n+1$ with weight
\[
w(\tilde{T}')=\frac{\psi_{c(v)+1}}{\psi_{c(v)}}\cdot w(T').
\]
Analogous considerations are valid when attaching label $n+1$ to the tree $T''$ obtaining a tree $\tilde{T}''$.
On the other hand, we can start with random bucket increasing trees $T',T''$ of size $n$ chosen due to the random tree model and attach label $n+1$ according to the probabilities $p_{T'}(v)$ and $p_{T''}(v)$. This leads to the following probabilities
of obtaining trees $\tilde{T}'$ and $\tilde{T}''$:
\[
\P\{\tilde{T}'\}=
\P\{T'\}\cdot 
\begin{cases}
p_{T'}(v), & \quad c(v)<b,\\
\frac{p_{T'}(v)}{\deg(v)+1}, & \quad c(v)=b,
\end{cases}
\]
and
\[
\P\{\tilde{T}''\}=
\P\{T''\}\cdot 
\begin{cases}
p_{T''}(v), & \quad c(v)<b,\\
\frac{p_{T''}(v)}{\deg(v)+1}, & \quad c(v)=b.
\end{cases}
\] 
Since the resulting trees must be random bucket increasing trees of size $n+1$, such that 
\[
\P\{\tilde{T}'\}=\frac{w(\tilde{T}')}{T_n},\quad
\P\{\tilde{T}''\}=\frac{w(\tilde{T}'')}{T_n},
\]
consequently it holds
\[
\frac{\P\{\tilde{T}'\}}{\P\{\tilde{T}''\}}
=\frac{w(\tilde{T}')}{w(\tilde{T}')}.
\]
This leads to four different equations, 
distinguishing between the capacities of the nodes $v'\in V(T')$ and $v''\in V(T'')$ attracting the label $n+1$. 
For $c(v')=c(v'')=b$ we get
\begin{align*}
\frac{\psi_1\frac{\varphi_{\deg(v')+1}}{\varphi_{\deg(v')}}\cdot w(T')}{\psi_1\frac{\varphi_{\deg(v'')+1}}{\varphi_{\deg(v'')}}\cdot w(T'')}=\frac{\frac{p_{T'}(v')}{\deg(v')+1}\cdot w(T')}{\frac{p_{T''}(v'')}{\deg(v'')+1}\cdot w(T'')}.
\end{align*}
For $c(v')<b$ and $c(v'')=b$ we get
\begin{align*}
\frac{\frac{\psi_{c(v')+1}}{\psi_{c(v')}}\cdot w(T')}{\psi_1\frac{\varphi_{\deg(v'')+1}}{\varphi_{\deg(v'')}}\cdot w(T'')}=\frac{p_{T'}(v')\cdot w(T')}{\frac{p_{T''}(v'')}{\deg(v'')+1}\cdot w(T'')},
\end{align*}
and a similar equation for $c(v')=b$ and $c(v'')<b$. The final case 
$c(v')<b$ and $c(v'')<b$ gives
\begin{align*}
\frac{\frac{\psi_{c(v')+1}}{\psi_{c(v')}}\cdot w(T')}{\frac{\psi_{c(v'')+1}}{\psi_{c(v'')}}\cdot w(T'')}=\frac{p_{T'}(v')\cdot w(T')}{p_{T''}(v'')\cdot w(T'')}.
\end{align*}
By considering vertices $v',u'\in V(T')$ with $c(v')=b$, $c(u')<b$, vertices $v'',u''\in V(T'')$ with $c(v'')=b$, $c(u'')<b$, and taking into account the four possible cases, we obtain, 
\begin{align*}
C_n := & \frac{1}{p_{T'}(v')}\cdot\psi_1\cdot\frac{(\deg(v')+1) \, \varphi_{\deg(v')+1}}{\varphi_{\deg(v')}}
=\frac{1}{p_{T'}(u')}\cdot\frac{\psi_{c(u')+1}}{\psi_{c(u')}}\\
&=\frac{1}{p_{T''}(v'')}\cdot\psi_1\cdot\frac{(\deg(v'')+1) \, \varphi_{\deg(v'')+1}}{\varphi_{\deg(v'')}}
=\frac{1}{p_{T''}(u'')}\cdot\frac{\psi_{c(u'')+1}}{\psi_{c(u'')}}.
\end{align*}
Multiplication with $p_{T'}(v')$ or $p_{T'}(u')$ and summing up over all vertices $v\in V(T')$ 
then gives 
\begin{align*}
C_n&=\sum_{v \in V(T')}p_{T'}(v)\cdot C_n=\sum_{k=1}^{b-1} m_{k} \frac{\psi_{k+1}}{\psi_{k}} + \sum_{k \ge 0} n_{k} (k+1) \psi_{1} \frac{\varphi_{k+1}}{\varphi_{k}},
\end{align*}
which is the stated balance condition.

\medskip

\textbf{Case (v)\impli (vi).} We consider families $\mathcal{T}$ with the property that, when starting with a random tree of $T\in\mathcal{T}$ of size $n$ and removing all labels larger than $j$, we obtain a random tree of $T'$ of size $j$.
By iterating the argument, it is sufficient to assume that, after removing label $n$ in a random tree $T$ of size $n$, we get a random tree of $T'$ of size $n - 1$. This randomness preserving property
can be described via the equation
\begin{equation}
\label{eq:v-vi:1}
\frac{w(T')}{w(T'')}=\frac{\sum_{T \in \mathcal{T} : T \xrightarrow{(n)} T'} w(T)}{\sum_{T \in \mathcal{T} : T \xrightarrow{(n)}  T''} w(T)},
\end{equation}
which must hold for all bucket ordered trees $T',T''\in\mathcal{T}$ of size $n-1$ and bucket ordered trees $T\in\mathcal{T}$ of size $n$. 
Here, $ T \xrightarrow{(n)} T'$ describes the fact that by removing node $n$ from $T$ we obtain $T'$. We assume now that $T'$ is obtained from $T$ by removing label $n$, which was either contained in a node $u$ or attached to a saturated node $v$. 
We obtain then an equivalent characterization by considering all possible trees $T\in\mathcal{T}$ such that $T \xrightarrow{(n)} T'$ or 
$T \xrightarrow{(n)} T''$ and their weights. 
The left hand side of the resulting equation is readily obtained from the definition of the weight of a tree,
\begin{align*}
\frac{w(T')}{w(T'')}&=\frac{\bigg(\prod_{u\in T'\colon c(u)<b} \psi_{c(u)} \bigg) \cdot \bigg(\prod_{v\in T'\colon c(v)=b}\varphi_{\grad(v)}\bigg)}{\bigg(\prod_{u\in T''\colon c(u)<b} \psi_{c(u)} \bigg) \cdot \bigg(\prod_{v\in T''\colon c(v)=b}\varphi_{\grad(v)}\bigg)},
\end{align*}
whereas the right hand side is obtained by considering all nodes in $T'$ and $T''$, respectively, and the change in the respective weights if label $n$ is attached,
\begin{align*}
\frac{\sum_{T \in \mathcal{T} : T \xrightarrow{(n)} T'} w(T)}{\sum_{T \in \mathcal{T} : T \xrightarrow{(n)}  T''} w(T)}
&=\frac{\bigg(\prod_{u\in T'\colon c(u)<b} \psi_{c(u)} \bigg) \cdot \bigg(\prod_{v\in T'\colon c(v)=b}\varphi_{\grad(v)}\bigg)
}{\bigg(\prod_{u\in T''\colon c(u)<b} \psi_{c(u)} \bigg) \cdot \bigg(\prod_{v\in T''\colon c(v)=b}\varphi_{\grad(v)}\bigg)}\\
&\ \times\frac{\bigg(\sum_{k=1}^{b-1} m_{k}(T') \frac{\psi_{k+1}}{\psi_{k}} + \sum_{k \ge 0} n_{k}(T') (k+1) \psi_{1} \frac{\varphi_{k+1}}{\varphi_{k}}\bigg)}{\bigg(\sum_{k=1}^{b-1} m_{k}(T'') \frac{\psi_{k+1}}{\psi_{k}} + \sum_{k \ge 0} n_{k}(T'') (k+1) \psi_{1} \frac{\varphi_{k+1}}{\varphi_{k}}\bigg)}.
\end{align*}
This equation holds for all ordered trees $T',T''\in\mathcal{T}$ of size $n-1$. Thus,~\eqref{eq:v-vi:1} implies the balance condition,
\begin{equation*}
  \sum_{k=1}^{b-1} m_{k} \frac{\psi_{k+1}}{\psi_{k}} + \sum_{k \ge 0} n_{k} (k+1) \psi_{1} \frac{\varphi_{k+1}}{\varphi_{k}} =: C_{n-1},
\end{equation*}
with $C_{n-1}$ being independent of the particular tree of size $n-1$.

\medskip

\textbf{Case (vi)\impli (i).} We determine, which bucket weight and degree weight sequences $(\psi_k)_{1 \le k \le b-1}$ and $(\varphi_k)_{k \ge 0}$ satisfy the balance equation. We define
\begin{equation}
\label{eq:A0}
\beta_k=\frac{\psi_{k+1}}{\psi_k},\ 1\le k\le b-1,\quad
\gamma_k=\psi_1\cdot(k+1)\frac{\varphi_{k+1}}{\varphi_{k}},\ k\ge 0,
\end{equation}
(recall that we set $\psi_{b} = \varphi_{0}$) and the balance equation gets the form
\[
C_n=\sum_{k=1}^{b-1} m_{k}\beta_k + \sum_{k \ge 0} n_{k} \gamma_k. 
\]
Next we consider specific shapes of bucket increasing trees to determine 
all sequences $(\beta_k)_{1 \le k \le b-1}$ and $(\gamma_k)_{k \ge 0}$, satisfying this equation.
First, we consider trees of size $n=b(b+1)$. For $1\le k\le b$ let $T=T(k)$ denote a chain of $b$ fully saturated nodes. 

\begin{figure}[!htb]
\includegraphics[height=6cm]{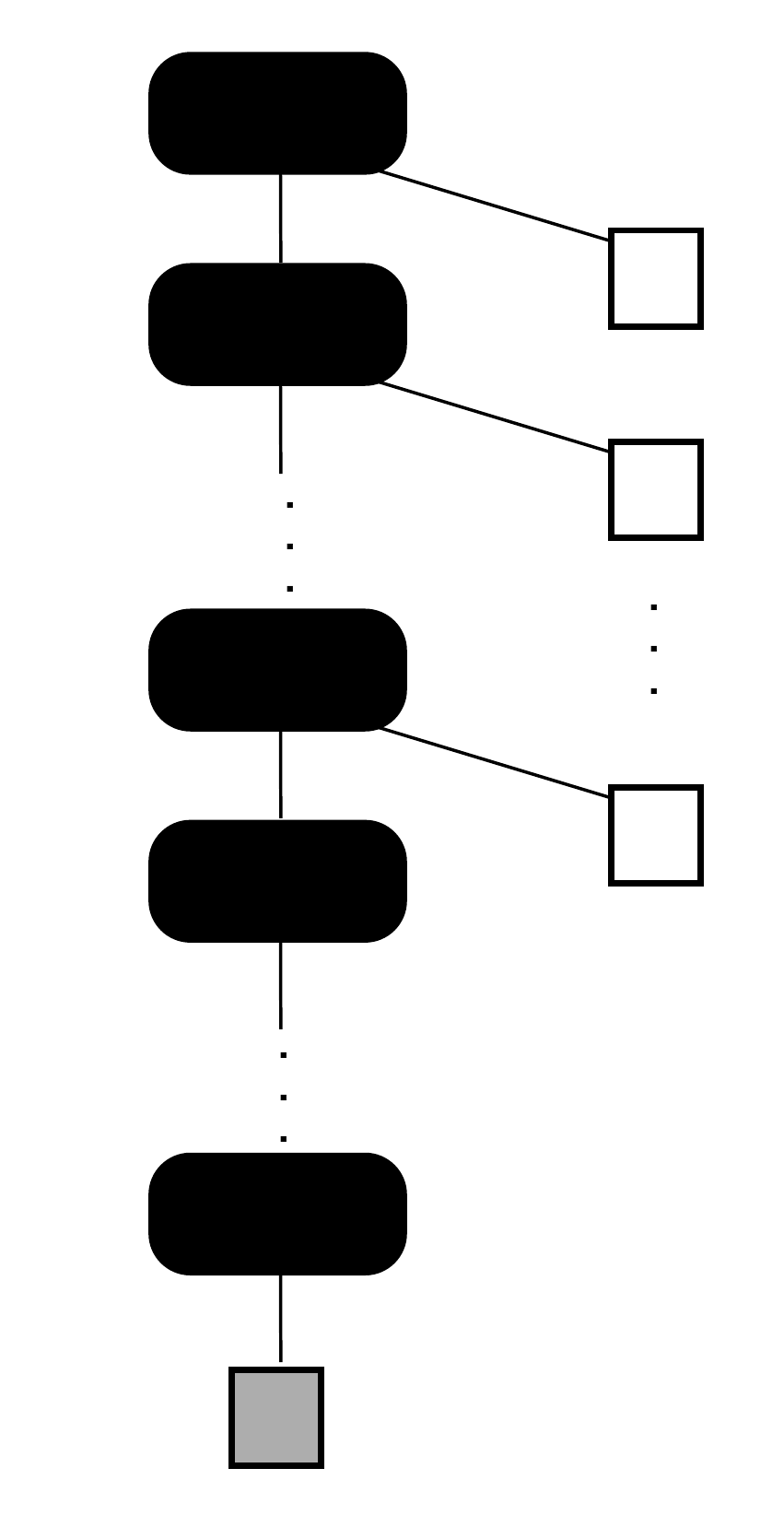}
\caption{Trees $T=T(k)$: $b$ black nodes with capacity $b$, $b-k$ white leaves of capacity $1$, and a single gray leaf of capacity $k$.}
\label{}%
\end{figure}

The first $b-k$ nodes additionally have attached to them a single label, and at the last node of the chain there is attached a node of capacity $c(v)=k$. In the case $k=b$ we have a tree of $b+1$ nodes of capacity $c(v)=b$.
The balance condition applied to these $b$ trees $T=T(k)$ yields the system of equations
\begin{equation}
\label{eq:A1}
\begin{split}
C_n&=(b-k)\beta_1 + \beta_k +(b-k)\gamma_2+k\gamma_1,\quad 1\le k\le b-1 \\
C_n&= b \gamma_1+\gamma_0,\quad k=b.
\end{split}
\end{equation}
In particular, for $k=1$ we get
\[
C_n=b\beta_1 + (b-1)\gamma_2+\gamma_1.
\]
Combining this equation with~\eqref{eq:A1}
gives the equation
\[
k\beta_1 - \beta_k +(k-1)\gamma_2-(k-1)\gamma_1=0,
\]
leading to the recurrence relation
\begin{equation}
\label{eq:A2}
\beta_k=k(\beta_1 +\gamma_2-\gamma_1)+\gamma_1-\gamma_2,\quad 1\le k\le b-1.
\end{equation}
Moreover, evaluating~\eqref{eq:A1} at $k=1$ leads to
\begin{equation}
\label{eq:A3}
\gamma_0
=b(\beta_1-\gamma_1+\gamma_2)+\gamma_1-\gamma_2.
\end{equation}

\smallskip

\begin{figure}[!htb]
\begin{center}
\begin{minipage}{10cm}
\includegraphics[height=4cm]{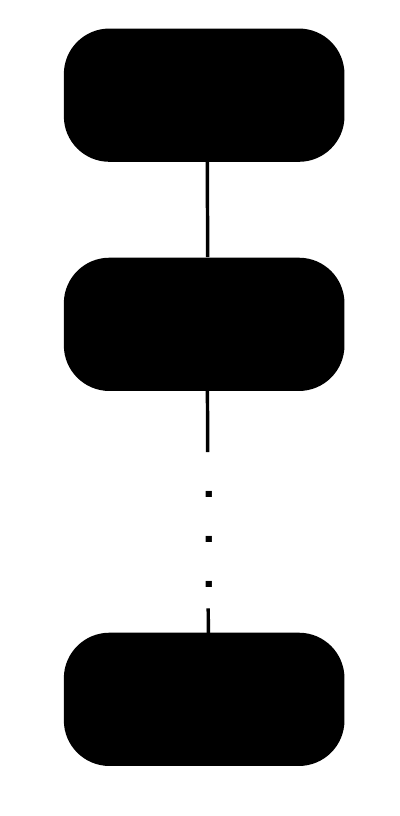}
\hspace{1cm}
\includegraphics[scale=0.5]{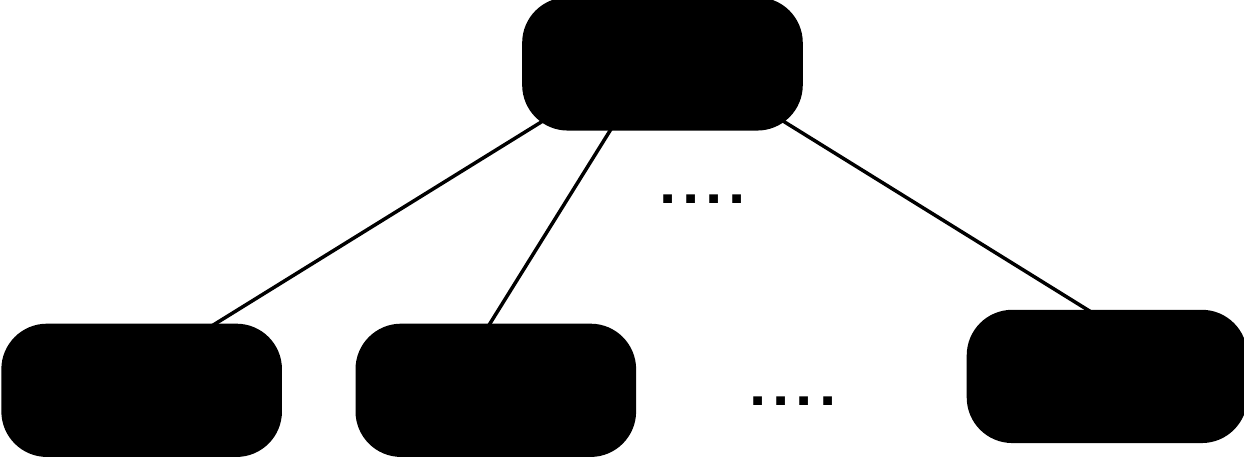}
\end{minipage}
\caption{Two shapes of trees: chains $K$ of $k+1$ fully saturated nodes and stars $S$ of $k+1$ fully saturated nodes.\label{DiamondFig1} }
\end{center}
\end{figure}

Next we consider two different tree shapes, namely the following two infinite sequences of trees. We look at stars $S=S(k+1)$ of size $n=(k+1)b$, consisting of a root plus $k$ saturated nodes attached to the root, and chains $K=K(k+1)$ of size $n=(k+1)b$, consisting of $k$ saturated nodes attached to each other.   
The stars lead to the balance equation
\begin{equation}
\label{eq:A4}
\gamma_k+k\gamma_0=C(n),
\end{equation}
whereas the chains gives
\begin{equation}
\label{eq:A5}
k\gamma_1+\gamma_0=C(n).
\end{equation}
Combining~\eqref{eq:A4} and \eqref{eq:A5} leads to the recurrence relation
\begin{equation}
\label{eq:A6}
\gamma_k=k(\gamma_1-\gamma_0)+\gamma_0,\quad k\ge 0.
\end{equation}

In particular we obtain for $k=2$ the equation
\[
\gamma_2=2(\gamma_1-\gamma_0)+\gamma_0,
\]
which implies that
\[
\gamma_2-\gamma_1=\gamma_1-\gamma_0.
\]
Using this relation, we can simplify equations~\eqref{eq:A2} and~\eqref{eq:A3} obtaining
\begin{equation}
\begin{split}
\label{eq:A7}
\beta_k&=k(\beta_1 +\gamma_1-\gamma_0)+\gamma_0-\gamma_1,\quad 1\le k\le b-1,\\
\gamma_0&=b(\beta_1+\gamma_1-\gamma_0)+\gamma_0-\gamma_1.
\end{split}
\end{equation}
The latter equation allows to express $\beta_1$ in terms of $\gamma_0$ and $\gamma_1$:
\begin{equation*}
\beta_1=\frac{1}{b}\cdot\gamma_1+\gamma_0-\gamma_1,
\end{equation*}
as well as to obtain 
\[
\beta_1+\gamma_1-\gamma_0=\frac{1}{b}\cdot\gamma_1.
\]
From~\eqref{eq:A6} and~\eqref{eq:A7} we finally obtain the characterizing equations
\begin{equation}
\begin{split}
\label{eq:A8}
\beta_k&=\frac{k}{b}\cdot\gamma_1+\gamma_0-\gamma_1,\quad 1\le k\le b-1,\\
\gamma_k&=k(\gamma_1-\gamma_0)+\gamma_0,\quad k\ge 0.
\end{split}
\end{equation}
We solve \eqref{eq:A8} and distinguish according to the sign of $\gamma_1-\gamma_0$ obtaining three different cases.
\begin{itemize}
\item First, let $\gamma_1-\gamma_0=0$. We set $\gamma_0=a$ and obtain
\[
\gamma_k=a,\quad \beta_k=k\cdot\frac{a}{b}.
\]
From~\eqref{eq:A0} we obtain 
\[
\psi_{k+1}=k\cdot\frac{a}{b}\psi_k,\ 1\le k\le b-1,\quad
\varphi_{k+1}=\frac{a}{(k+1)\psi_1} \varphi_k,\ k\ge 0,
\]
leading to 
\begin{equation*}
\begin{split}
\psi_k & = (k-1)!\frac{a^{k-1}}{b^{k-1}}\psi_1,\quad 1 \le k \le b-1,\\ 
\varphi_k & = \frac{a^k}{k!\psi_1^k}\varphi_0=\frac{a^k}{k!\psi_1^k}\cdot(b-1)!\frac{a^{b-1}}{b^{b-1}}\psi_1, \quad k\ge 0.
\end{split}
\end{equation*}
Thus, this case gives the family of bucket recursive trees, where setting $a=b$ and $\psi_1=1$ yields the representation used in Definition~\ref{def0}.

\smallskip

\item Next, we assume that $\gamma_1-\gamma_0<0$. 
Due to our assumption $\varphi_k\ge 0$, there has to exist a $D\in\N$ 
such that $\gamma_D=D(\gamma_1-\gamma_0)+\gamma_0=0$. In the following we exclude the degenerate case $D=1$, leading to chains.
We get $\gamma_1=\frac{D-1}{D}\gamma_0$. We set again $\gamma_0=a$, leading to
\[
\gamma_k=\frac{a}{D}(D-k),
\]
as well as
\[
\beta_k=\frac{a(D-1)}{b D }\cdot\Big(\frac{b}{D-1}+k\Big).
\]
By~\eqref{eq:A0} we obtain the recurrence relations for $\psi_k$ and $\varphi_k$ leading to
\begin{align*}
\psi_k&=\frac{a^{k-1}(D-1)^{k-1}}{b^{k-1}D^{k-1}}\fallfak{\big(k-1+\frac{b}{D-1}\big)}{k-1}\cdot\psi_1, \quad 1\le k\le b,\\
\varphi_k&= \frac{a^k\fallfak{D}{k}}{D^k\psi_1^k k!}\cdot \psi_b, \quad k \ge 0.
\end{align*}
Setting $D=b(d-1)+1$ leads then to 
\begin{align*}
\psi_k&=\frac{a^{k-1}(d-1)^{k-1}}{(b(d-1)+1)^{k-1}}(k-1)!
\binom{k-1+\frac{1}{d-1}}{k-1}\cdot\psi_1, \quad 1\le k\le b,\\
\varphi_k&= \frac{a^{k}}{(b(d-1)+1)^{k}\psi_1^{k}}\binom{b(d-1)+1}{k}\\
& \qquad \times \frac{a^{b-1}d^{b-1}}{(b(d-1)+1)^{b-1}}(b-1)!\binom{b-1+\frac{1}{d-1}}{b-1}, \quad k \ge 0.
\end{align*}
Thus, this case yields the family of \dit, where specializing $\psi_1=1$ and $a=b(d-1)+1$ leads to the representation given in Definition~\ref{def1}.

\smallskip

\item Finally, we assume that $\gamma_1-\gamma_0>0$. We set $c=\gamma_1-\gamma_0$ and again $\gamma_0=a$. Proceeding as before first gives 
\begin{align*}
\gamma_k=c\Big(k+\frac{a}{c}\Big),\quad \beta_k=\frac{a+c}{b}\cdot\Big(k-\frac{bc}{a+c}\Big).
\end{align*}
Note that $\beta_1>0$ implies
\[
1-\frac{bc}{a+c}>0 \quad \text{and} \quad \frac{a+c}{bc}>1.
\]
Due to \eqref{eq:A0} we obtain recurrence relations for $\psi_k$ and $\varphi_k$, from which we finally obtain
\begin{align*}
\psi_k&= \frac{(a+c)^{k-1}}{b^{k-1}}\cdot (k-1)!\cdot\binom{k-1-\frac{bc}{a+c}}{k-1}\cdot\psi_1, \quad 1\le k\le b-1\\
\varphi_k&=\frac{c^k}{\psi_1^k}\binom{k-1+\frac{a}{c}}{k}\cdot \frac{(a+c)^{b-1}}{b^{b-1}}(b-1)!\binom{b-1-\frac{bc}{a+c}}{b-1}\cdot\psi_1, \quad k\ge 0.
\end{align*}
Thus, this case gives the family of \port, where specializing $\frac{a+c}{bc}=\alpha+1$ with $\alpha>0$ and additionally $\psi_1=c=1$ leads to the representation used in Definition~\ref{def2}.
\end{itemize}
\end{proof}

\section{Applications and Outlook}
We finish this work by presenting immediate applications and discussing a few lines of further research.

\subsection{Applications}
In the following we discuss the random variable $Y_{n,j}$, which counts the number of descendants of element $j$, i.e., the total number of elements with a label greater or equal $j$ contained in the subtree rooted with the bucket containing element $j$, in a random bucket increasing tree of size $n$ (with maximal bucket size $b$). For this random variable,
the exact distribution, limit laws, as well as a decomposition of the random variable of interest in terms of the initial bucket size $K_j$ (i.e., the size of the bucket containing label $j$ after its insertion)
was provided in~\cite{BucketPanKu2009,BucketPanKuAccepted}, generalizing earlier results for ordinary increasing trees~\cite{IncDesc}. In particular, phase transitions of the limit law depending of the growth of $j=j(n)$ with respect to the total number of labels $n$ were observed. Below we provide a much more detailed insight into the distribution of $Y_{n,j}$, for fixed $j\ge b+1$, refining the previously observed beta limit law for $Y_{n,j}$. Note that for $1\le j\le b$ the number of descendants naturally degenerates: $Y_{n,j}=n+1-j$. Crucially for this refinement are the growth processes presented, which generate bucket increasing trees in a step-by-step fashion. In turn, we can generalize the correspondence between descendants in increasing trees and so-called P\'olya-Eggenberger urn models (see \cite{KuPa2014} and references therein) to bucket increasing trees.

\smallskip

We consider label $j$ during the respective tree evolution process generating random trees of the tree families studied. 
We distinguish between two events: either a new label is attached to the subtree rooted at the bucket containing $j$ (for short, attached to the subtree rooted $j$) or it is attached elsewhere. These two possibilities translate into the ball replacement matrix of a two-colour urn model.
In contrast to ordinary increasing trees, case $b=1$, the urn model has for $b>1$ already random initial configurations
as the size $K_j$ of the bucket containing label $j$ at the insertion of this label is a random variable itself. The support 
of $K_j$ is given in terms of the bucket size $b$ and equals $\{1,2,\dots,b\}$; see~\cite{BucketPanKuAccepted} for the probability mass function of $K_j$. In the following we outline the procedure for {\dit} taking into account Definition~\ref{def1} (the other tree families can be treated in a similar way). Consider such a $(b,d)$-ary bucket increasing tree of size $j$: there are $1+(d-1)j$ possible attachment positions, where a new label can be inserted. Exactly $1+(d-1)K_{j}$ such positions are contained in the subtree rooted at label $j$, whereas the other $(d-1)(j-K_{j})$ are not. In the urn model description we will use balls of two colours, black and white. Each white ball will correspond to a position contained in the subtree rooted $j$, whereas each black ball will correspond to a position that is not contained in the subtree rooted $j$. This already describes the initial conditions of the urn. During the tree evolution process, when attaching a new node to a position in the subtree of $j$, then there appear $d-1$ new positions in this subtree, whereas when attaching a new node to a position in the remaining tree, then there appear $d-1$ new positions in the remaining tree. In the urn model description this simply means that when drawing a white ball one adds $d-1$ white balls and when drawing a black ball one adds $d-1$ black balls to the urn. After $n-j$ draws, which correspond to the $n-j$ attachments of nodes in the tree, the number of white balls in the urn is linearly related to the size of the subtree rooted $j$ in a tree of size $n$.
Thus, the following urn model description immediately follows.
\begin{urn}[Urn models for descendants of label $j$]
Consider a balanced P\'olya urn with ball replacement matrix
\[
M = \left(
\begin{matrix}
\sigma & 0\\
0 & \sigma
\end{matrix}
\right), \qquad \text{where} \quad 
\sigma =
\begin{cases}
1, & \text{for } \Rec,\\
d-1, & \text{for } \dit,\\
\alpha+1, & \text{for } \port,\\
\end{cases}
\]
and initial conditions determined by the random initial bucket size $K_j$:
\[
W_0= 
\begin{cases}
K_j, \\ 
(d-1)K_j+1,\\ 
(\alpha+1)K_j-1,\\  
\end{cases}
\qquad 
B_0= 
\begin{cases}
j-K_j, & \text{for } \Rec,\\ 
(d-1)(j-K_j), & \text{for } \dit,\\ 
(\alpha+1)(j-K_j), & \text{for } \port,\\
\end{cases}
\]
with $1\le j\le n$. The number $Y_{n,j}$ of descendants of node $j$ in an increasing tree of size $n$ has the same distribution as the (shifted and scaled) number of white balls $W_{n-j}$ in the P\'olya urn after $n-j$ draws, i.e.,
\[
Y_{n,j}\law 
\begin{cases}
W_{n-j}-K_{j}+1, & \text{for } \Rec,\\
(W_{n-j}-1)/\sigma-K_{j}+1, & \text{for } \dit,\\
(W_{n-j}+1)/\sigma-K_{j}+1, & \text{for } \port.
\end{cases}
\]
\end{urn}

Let us consider $Y_{n,j}$ for fixed $j$. It is already known~\cite{IncDesc,BucketPanKu2009,BucketPanKuAccepted} that 
\begin{equation}\label{eqn:Ynj_Beta}
\frac{Y_{n,j}}{n}\claw \beta(K_j + \kappa, j - K_j), \quad \text{with} \quad 
\kappa=
\begin{cases}
0, & \text{for } \Rec,\\
\frac{1}{d-1}, & \text{for } \dit,\\
-\frac{1}{\alpha+1}, & \text{for } \port,
\end{cases}
\end{equation}
which we denote by
\begin{equation}
\label{LL1}
Y_{n,j} \sim n\cdot \beta(K_j + \kappa, j - K_j).
\end{equation}
We note that the beta limit law and the distributional convergence can readily be reobtained and strengthened in a few lines using 
discrete martingales. Let $N\ge 1$ denote the discrete time and $\mathcal{F}_{N}$ the $\sigma$-algebra generated by the first $N$ draws from the urn. Then 
\[
\E(W_N | \mathcal{F}_{N-1})=W_{N-1} + \sigma\frac{W_{N-1}}{T_{N-1}}=\frac{T_N}{T_{N-1}}\cdot W_{N-1},
\]
where the total number of balls $T_N=W_N+B_N$ after $N$ draws is given by $T_N=\sigma N +T_0$.
Consequently, $\mathcal{W}_N=W_N/T_N$ is a non-negative martingale and converges almost surely. This further implies that $W_N/(\sigma N)$ also converges almost surely.
Moreover, for integers $s\ge 1$, the binomial moments $\binom{W_N/\sigma+s-1}{s}$ satisfy
\begin{equation*}
\begin{split}
\E\bigg(\left.\binom{\frac{W_N}{\sigma}+s-1}{s}\right|\mathcal{F}_{N-1}\bigg)&=\binom{\frac{W_{N-1}}{\sigma}+s-1}{s} +\frac{W_{N-1}}{T_{N-1}}\cdot \binom{\frac{W_{N-1}}{\sigma}+s-1}{s-1}\\
&=\binom{\frac{W_{N-1}}{\sigma}+s-1}{s}\cdot \frac{T_{N+s-1}}{T_{N-1}}.
\end{split}
\end{equation*}
Consequently, 
\[
\mathcal{W}_{N,s}=\frac{\binom{\frac{W_N}{\sigma}+s-1}{s}}{T_{N+s-1}\cdots T_N}
\]
is a martingale and we also obtain the moments
\[
\E\bigg(\binom{\frac{W_N}{\sigma}+s-1}{s}\bigg)=\binom{\frac{W_0}{\sigma}+s-1}{s}\cdot\frac{T_{N+s-1}\cdots T_N}{T_{s-1}\cdots T_0},
\]
such that, for $N\to\infty$, it holds
\[
\E(W_N^s)\sim \sigma^s\cdot\frac{N^s\binom{\frac{W_0}\sigma +s-1}{s}}{\binom{\frac{T_0}\sigma+s-1}{s}}.
\]
This directly leads to the beta limit law via an application of the method of moments.

\smallskip

Let $\hat{Y}_{n,j}^{[K]}$ denote the conditional version of $Y_{n,j}$ on the event $\{K_j=K\}$, with $K\in\{1,\dots, b\}$.
As mentioned before, the limit laws of the urn model can be translated directly to gain a refinement of~\eqref{LL1} establishing almost-sure convergence. Moreover, one may use again discrete martingales to study the difference of $\hat{Y}_{n,j}^{[K]}$ and $n\cdot\beta$, with $\beta=\beta(K + \kappa, j - K)$ denoting the almost sure beta limit law and $\kappa$ given in~\eqref{eqn:Ynj_Beta}. Such random centerings frequently occur in martingale theory and are called martingale tail sums. The results of Heyde~\cite{Heyde1977}, see also~\cite{hallheyde80} and \cite{HK}, imply the following refinement of the beta limit law~\cite{IncDesc,BucketPanKu2009,BucketPanKuAccepted}:
\begin{equation}
\label{LL2}
\hat{Y}_{n,j}^{[K]} \sim n\cdot \beta+c\cdot \sqrt{\beta(T_0-\beta)}\sqrt{n} \, \mathcal{N},
\end{equation}
where $c$ denotes a constant, $\beta=\beta(K + \kappa, j - K)$ the almost sure beta limit law and $\mathcal{N}$ a standard Gaussian random variable. 
Written in a more conventional notation, we gain, for $n\to\infty$, the following convergence in distribution result:
\[
\sqrt{n}\cdot \Big(\frac{\hat{Y}_{n,j}^{[K]}}{n}-\beta\Big)\claw c\cdot \sqrt{\beta(T_0-\beta)} \, \mathcal{N}.
\]
Moreover, a law of the iterated logarithm for $\hat{Y}_{n,j}^{[K]}$ can be obtained also by a direct translation of corresponding results for the respective urn models~\cite{Heyde1977,HK}. Finally, we note in passing that the outdegree distribution of label $j$ in {\port} can be treated in a similar manner relying on the precise second order results for so-called triangular urn models~\cite{HK}.

\subsection{Outlook}
Split trees and bucket increasing trees: Random split trees were introduced by Devroye~\cite{DevroyeA} as rooted trees generated
by a certain recursive procedure using a stream of balls added to the root. This model encompasses a great many other tree models, 
including binary search trees, $m$-ary search trees, as well as $d$-ary increasing trees. Recently, Janson~\cite{JansonB} has shown
that also recursive trees and generalized plane-oriented recursive trees can be modeled in terms of split trees.
We expect that all three families of bucket increasing trees generated by a tree evolution process are also random split trees.
We note in passing a two-stage process to construct them: first, generate ordinary recursive trees, $d$-ary increasing trees 
and generalized plane-oriented recursive trees as outlined in~\cite{DevroyeA,JansonB} and second, apply to them a clustering map described in~\cite{BucketPanKuAccepted}.

\smallskip

Bilabelled trees: In particular for bilabelled trees, bucket size $b=2$, these findings open the gate for a combined approach: a combinatorial analysis and a top-down approach using the underlying tree structure and tools from analytic combinatorics, as well as a bottom-up analysis using the step-by-step construction from the probabilistic growth rules. This also has applications for increasing diamonds~\cite{Hwang2016}, which are certain directed acyclic graphs useful for modelling executions of series-parallel concurrent processes. As we know from~\cite{BucketPanKuAccepted}, increasing diamonds are in bijection with bilabelled increasing trees. Thus, our results for bucket size $b=2$ allow to study families of increasing diamonds using probabilistic growth rules. The authors are currently investigating into this matter. 

\smallskip

Bucket size $b=b(n)$: One can also study the effects of bucket sizes $b=b(n)$, depending on the total number of labels $n$, as $n$ tends to infinity. How do parameters like depths, profiles, height, width, etc. and their limit laws change when $b\to \infty$? At which growth rate $b=b(n)$ is the root degree unbounded and all other nodes have bounded degrees? Various questions of such kind seem to be of interest.

\section{Acknowledgments}
We warmly thank a referee of~\cite{BucketPanKuAccepted} for emphasizing the problem of determining all tree evolution processes creating random bucket increasing trees, leading to the current work.

\bibliography{bucketinc-evo-refs}{}
\bibliographystyle{plain}

\end{document}